\documentclass[12pt,a4paper]{article}
\usepackage{amsmath, amssymb}
\usepackage{amsthm}
\usepackage{url}

\newtheorem{theorem}{Theorem}[section]
\newtheorem{corollary}[theorem]{Corollary}
\newtheorem{lemma}[theorem]{Lemma}
\newtheorem{proposition}[theorem]{Proposition}

\theoremstyle{definition}
\newtheorem{definition}[theorem]{Definition}

\theoremstyle{remark}
\newtheorem{remark}{Remark}
\newtheorem{example}{Example}

\newcommand{\norm}  [1]{\ensuremath{\left  \|      #1  \right \|       }}

\newcommand{\sqb}   [1]{\ensuremath{\left [        #1  \right ]        }}

\newcommand{\cb}    [1]{\ensuremath{\left  \{      #1  \right \}       }}

\newcommand{\of}    [1]{\ensuremath{\left (        #1  \right )        }}

\newcommand{\scp}   [1]{\ensuremath{\left \langle  #1  \right \rangle  }}

\newcommand{\st} {\ensuremath{\; | \;}}

\newcommand{\X}{\mathcal{X}}

\newcommand{\C}{\mathcal{C}}

\newcommand{\D}{\mathcal{D}}

\newcommand{\R}{\mathrm{I\negthinspace R}}

\renewcommand{\P}{\mathcal{P}}
\newcommand{\N}{\mathrm{I\negthinspace N}}
\newcommand{\cl}{{\rm cl\,}}
\newcommand{\cone}{{\rm cone\,}}
\newcommand{\Int}{{\rm int\,}}

\newcommand{\leer}[1]{}
\DeclareMathOperator{\wMin}{wMin}

\DeclareMathOperator{\ri}{ri}

\DeclareMathOperator*{\id}{ id}

\DeclareMathOperator*{\Min}{ Min}
\DeclareMathOperator*{\rMin}{ rMin}
\DeclareMathOperator*{\epi}{ epi}
\DeclareMathOperator*{\dom}{ dom}
\DeclareMathOperator*{\ind}{ Ind}
\DeclareMathOperator*{\Liminf}{ Liminf}
\DeclareMathOperator*{\Limsup}{ Limsup}
\DeclareMathOperator*{\conv}{ conv}

\newcommand{\OLR}{\overline{\mathrm{I\negthinspace R}}}

\title{Geometric Duality for Convex Vector Optimization Problems}
\author{Frank Heyde\thanks{University of Graz, Institute of Mathematics and Scientific Computing, Heinrichstr. 36, A-8020 Graz, Austria (frank.heyde@uni-graz.at)}}

\begin{document}
\maketitle

\begin{abstract}
Geometric duality theory for multiple objective linear programming problems turned out to be very useful for the development of efficient algorithms to generate or approximate the whole set of nondominated points in the outcome space. This article extends the geometric duality theory to convex vector optimization problems.\\[1em]
\end{abstract}

\section{Introduction}

Recently, a geometric duality theory for multiple objective linear programming problems was introduced in 
\cite{HeyLoe2008}. This theory deals with a duality relation between
the polyhedral extended image sets of a primal and a dual vector optimization
problem, which is similar to duality of polytopes, by providing an inclusion reversing
one-to-one map between the set of all maximal proper faces of the dual image and the set
of all weakly minimal proper faces of the primal image. Moreover, the dimensions of the corresponding faces of the 
primal and dual image are complementary in the sense that they always add up to the same value.

This kind of duality theory proved to be very useful in developing algorithms to generate or approximate the whole 
set of (weakly) minimal points of a vector optimization problem in the objective space. Ehrgott et al. 
\cite{ELS11} used geometric duality in order to obtain slight improvements of Benson's outer approximation 
algorithm and to develop a dual variant of that algorithm. They point out that algorithms working in the outcome 
space are often much faster than algorithms working in the decision space as, in typical applications, the dimension of the outcome space 
is much smaller than the dimension of the decision space (there are fewer objectives than variables). 
In L\"ohne's book \cite{Loehne2011} a detailed description of the algorithms and some extensions can be found.

L\"ohne and Rudloff \cite{LoeRud11} point out that geometric duality also plays a role in Mathematical Finance. In 
fact, the superhedging portfolios in markets with transaction costs can be computed by solving a sequence of linear 
vector optimization problems. L\"ohne and Rudloff introduce an algorithm for solving them based on Benson's outer 
approximation algorithm and they show that this algorithm is related to existing ones via geometric duality.

In the present article, the geometric duality theory will be generalized to vector optimization problems where the 
extended image sets don't need to be polyhedral, they merely need to be convex and satisfy some second-order 
subdifferentiability condition. Moreover, this theory can deal with a general nontrivial ordering cone as needed, 
e.g., in \cite{LoeRud11}. The ordering cones neither need to be polyhedral nor need to have nonempty interior. 

The paper is structured as follows.
Section \ref{sect:prelim} collects some preliminaries from convex analysis, about minimal points with respect to a vector preorder and faces of convex sets.
In order to construct the inclusion reversing one-to-one map we show in section \ref{sect:epi} how such a map between the minimal faces of the epigraph of a proper closed convex function $f$ and the minimal faces of the epigraph of its Legendre-Fenchel conjugate $f^*$ can be obtained.
Section \ref{sect:secondorder} shows how a polarity relation between the second-order subdifferentials of $f$ and $f^*$ generalizes the complementarity condition for the dimensions of the corresponding faces in the linear case.
Finally, we show in section \ref{sect:vectopt} how the extended image of a convex vector optimization problem can be transformed into the epigraph of a suitable function $f$, the dual problem will be derived by an appropriate transformation of the epigraph of $f^*$ and geometric duality relations between the primal and the dual problem will be derived from the results in the previous sections. Moreover, we derive geometric duality relations for linear vector optimization problems with general ordering cone, that slightly extend the results in \cite{HeyLoe2008}, as a special case of the general theory.

\section{Preliminaries}\label{sect:prelim}

\subsection{Convex Analysis}

First, we will shortly collect the basic concepts and results of convex analysis for extended real valued functions. For further reference the reader is advised to consult any standard text book about convex analysis (e.g., \cite{Rockafellar72}, \cite{Zalinescu}).

Let $f:\R^n \to \OLR:=\R\cup\cb{+\infty,-\infty}$ be an extended real-valued function.
The set
\[
\dom f := \cb{x\in \R^n \st f(x) \neq +\infty}
\]
is called the {\em domain} of $f$ and the set 
\[
\epi f := \cb{(x,r)\in\R^n\times \R \st f(x) \le r}
\]
is called the {\em epigraph} of $f$. 
A function $f: \R^n \to \OLR$ is called {\em convex} if $\epi f$ is a convex set, $f$ is called {\em closed} if $\epi f$ is a closed set. Moreover, $f$ is called {\em proper} if $\dom f \neq\emptyset$ and
$f(x)\neq -\infty$ for all $x\in\R^n$.  

The {\em Legendre-Fenchel conjugate} of $f$, a function $f^*:\R^n \to \OLR$, is defined as 
\[
f^*(u):=\sup_{x\in \R^n}\of{\scp{x,u}-f(x)}
\]
where $\scp{\cdot,\cdot}$ denotes the standard Euklidean inner product in $\R^n$. 
The function $f^*$ is always a closed and convex function. If $f$ is a proper closed convex function then $f^*$ is proper as well and $(f^*)^*=f$.
Moreover, if $f$ is proper the inequality $\scp{x,u}\le f(x) + f^*(u)$, called Young-Fenchel inequality, holds true for all $x,u \in \R^n$.

The {\em directional derivative} of a proper convex function $f$ at a point $x_0\in\dom f$ in direction $x\in\R^n$ is defined as
\[
f'(x_0;x):= \inf\cb{\frac{f(x_0+tx)-f(x_0)}{t} \st t>0}.
\]
The {\em subdifferential} of a proper convex function $f$ at a point $x_0\in\dom f$, a subset $\partial f(x_0)$ of $\R^n$, can be defined in three equivalent ways
\begin{align*}
\partial f(x_0) &= \cb{u\in\R^n \st \forall x\in \R^n : f(x) \ge f(x_0) + \scp{u,x-x_0}} \\
								&= \cb{u\in\R^n \st \forall x\in \R^n : \scp{u,x}\le f'(x_0;x)} \\
								&= \cb{u\in\R^n \st \scp{x_0,u} = f(x_0)+f^*(u)}.
\end{align*}
If $f$ is proper closed convex, then from $(f^*)^*=f$ and the last characterization of $\partial f$ one easily concludes that
\[u\in \partial f(x) \qquad \Longleftrightarrow \qquad x\in\partial f^*(u).\]

The {\em normal cone} of a convex subset $A\subseteq \R^n$ at a point $x_0\in A$ is defined by
\[
\mathcal{N}_A(x_0):=\cb{u\in\R^n \st \forall x\in A : \scp{u,x-x_0} \le 0}.
\]
The subdifferential of a proper convex function $f$ at a point $x_0\in \dom f$ can be characterized by the normal cone of $\epi f$ at the point $(x_0,f(x_0))$ in the following way
\[
u \in \partial f(x_0) \quad\Leftrightarrow\quad (u,-1)\in\mathcal{N}_{\epi f}(x_0,f(x_0)).
\]
The {\em polar} of a set $A\subseteq \R^n$ is the set $A^\circ \subseteq \R^n$ defined by
\[
A^\circ = \cb{u\in \R^n \st \forall x\in A : \scp{u,x}\le 1}.
\]
The set $A^\circ$ is always a closed convex set containing the origin. We have $(A^\circ)^\circ = A$ if and only if $A$ is a closed convex set containing the origin.

\subsection{Supporting Hyperplanes and Exposed Faces of a Convex Set}

Let $v\in\R^q\setminus\cb{0}$ and $\alpha\in \R$. The set $H(v,\alpha):=\cb{y\in\R^q \st \scp{v,y}=\alpha}$ is a hyperplane in $\R^q$. We say that $H(v,\alpha)$ is a {\em supporting hyperplane} to a set $A\subseteq\R^q$ iff $A\cap H(v,\alpha)\neq\emptyset$ and $A\subseteq\cb{y\in \R^q \st \scp{v,y}\le \alpha}$. 

Let $A\subseteq \R^q$ be a convex set. A convex subset $F\subseteq A$ is called a {\em face} of $A$ if
\[
\of{y^1,y^2\in A ,\quad \lambda\in (0,1),\quad \lambda y^1 + (1-\lambda) y^2\in F} \quad\Rightarrow\quad y^1,y^2\in F.
\]
A face $F$ of $A$ is called {\em proper} if $\emptyset\neq F\neq A$.
A set $E\subseteq A$ is called an {\em exposed face} of $A$ if there is a supporting hyperplane $H(v,\alpha)$ to $A$ such that $E =H(v,\alpha)\cap A$. If $\dim A = q$ then each exposed face of a convex set $A$ is a proper face of $A$ as well. For polyhedral convex sets also the converse is true.

\subsection{Minimal and Weakly Minimal Points}

Let $C\subseteq \R^q$ be a closed convex cone.
We say that $y\in A$ is a minimal point of $A\subseteq \R^q$ with respect to $C$ if $\of{\cb{y}-C\setminus (-C)}\cap A =\emptyset$. The set of all minimal points of a set $A$ with respect to $C$ is denoted by $\Min_C A$, i.e.,
\[
\Min\nolimits_{C}A := \cb{y\in A \st \of{\cb{y}-C\setminus (-C)}\cap A =\emptyset}.
\]
If $C$ has nonempty interior then we say that $y\in A$ is a weakly minimal point of $A\subseteq \R^q$ with respect to $C$ if $\of{\cb{y}-\Int C}\cap A =\emptyset$. The set of all weakly minimal points of a set $A$ with respect to $C$ is denoted by $\wMin_C A$, i.e.,
\[
\wMin\nolimits_{C}A := \cb{y\in A \st \of{\cb{y}-\Int C}\cap A =\emptyset}.
\]

\section{Geometric Duality Map for Epigraphs}\label{sect:epi}

Throughout this section we assume that $f:\R^n\to \OLR$ is a proper closed convex function and 
\[
K=\cb{(x,r)\in \R^n\times \R \st x=0,r\ge 0}.
\]
In this section we will show how an inclusion-reversing one-to-one map between $K$-minimal exposed faces of the epigraph of $f$ and of the epigraph of the Legendre-Fenchel conjugate $f^*$ can be obtained.
Here a proper face is called $K$-minimal if all of its points are minimal with respect to $K$.

Since exposed faces are obtained by supporting hyperplanes we will collect some properties of supporting hyperplanes to $\epi f$. 

\begin{lemma}\label{LemSuppHP}
(i) If $H(u,s,\alpha)$ is a supporting hyperplane to $\epi f$, then $s\le 0$.

(ii) If $H(u,s,\alpha)$ is a supporting hyperplane to $\epi f$, then $H(u,s,\alpha)\cap (\epi f)$ is $K$-minimal in $\epi f$ if and only if $s<0$.

(iii) $H(u,-1,\alpha)$ is a supporting hyperplane to $\epi f$ if and only if $\partial f^*(u)\neq\emptyset$ and $\alpha =f^*(u)$.
\end{lemma} 
\begin{proof}
(i) Let $(x,r)\in (\epi f)\cap H(u,s,\alpha)$, i.e., $\scp{x,u}+rs=\alpha$ and $f(x)\le r$. Then $(x,r+1)\in \epi f$ hence 
$\scp{x,u}+(r+1)s \le\alpha$ which in turn implies $s\le 0$.

(ii) Let $s<0$ and assume that there is some $(x,r)\in H(u,s,\alpha)\cap (\epi f)$ that is not $K$-minimal in $\epi f$. Then there exists some $\delta >0$ with $(x,r-\delta)\in \epi f$. $(x,r)\in H(u,s,\alpha)$ implies $\scp{x,u}+rs=\alpha$ hence $\scp{x,u}+(r-\delta)s>\alpha$, a contradiction to the supporting hyperplane property.

If, on the other hand, $s=0$ ($s>0$ is impossible due to (i)) and $(x,r)\in H(u,s,\alpha)\cap (\epi f)$. Then $(x,r+1)\in H(u,s,\alpha)\cap (\epi f)$ as well and $(x,r+1)$ is not $K$-minimal in $\epi f$.

(iii) $\partial f^*(u)\neq\emptyset$ and $\alpha =f^*(u)$ is equivalent to the existence of some $\bar{x}\in X$ with $\alpha =f^*(u)=\scp{\bar x,u}-f(\bar x)$ which in turn is equivalent to $H(u,-1,\alpha)$ being a supporting hyperplane to $\epi f$ due to the definition of $f^*$.
\end{proof}
\begin{proposition}\label{PropExpFaces}
A subset $F\subseteq\epi f$ is a $K$-minimal exposed face of $\epi f$ iff there is some $\bar u\in\dom f^*$ with $\partial f^*(\bar u)\neq \emptyset$ such that
\[
F=\cb{(x,f(x))\in \R^n\times \R \st \bar u\in\partial f(x)}.
\]
Moreover, $F^*$ is a $K$-minimal exposed face of $\epi f^*$ iff there is some $\bar{x}\in\dom f$ with $\partial f(\bar{x})\neq \emptyset$ such that
\[F^*=\cb{(u,f^*(u))\in\R^{n}\times \R \st u\in\partial f(\bar{x})}.\]
\end{proposition}
\begin{proof}
We have
\begin{align*}
\cb{(x,f(x))\in \R^n\times \R \st \bar u\in\partial f(x)}&=\cb{(x,r)\in \R^n\times \R \st r=f(x), \scp{x,\bar u}=f(x)+f^*(\bar u}\\
&=\cb{(x,r)\in \R^n\times \R \st r\ge f(x),\; \scp{x,\bar u}-r=f^*(\bar u)}\\
&=H(\bar u,-1,f^*(\bar u))\cap (\epi f)
\end{align*}
where the $\supseteq$-relation in the second equality follows from the Young-Fenchel inequality. Hence the first statement follows from Lemma \ref{LemSuppHP}. The second statement can be proven analogously taking into account that $u\in\partial f(x)$ iff $x\in \partial f^*(u)$ for a proper closed convex function $f$.
\end{proof}
\begin{theorem}\label{TGD}   
The mapping $\Psi : 2^{\R^{n+1}}\to2^{\R^{n+1}}$ defined by 
\[
\Psi (F^*):=\bigcap_{(u,f^*(u))\in F^*}\cb{(x,f(x))\in\R^n\times \R \st u\in\partial f(x)}.
\]
is an inclusion reversing one-to-one mapping between $K$-minimal exposed faces of $\epi f^*$ and $K$-minimal exposed faces of $\epi f$. Its inverse mapping is given by
\[
\Psi^* (F):=\bigcap_{(x,f(x))\in F}\cb{(u,f^*(u))\in\R^n\times \R \st u\in\partial f(x)}
\] 
\end{theorem}
\begin{proof}
(a) The mapping is inclusion reversing by definition.

(b) We will show that $\Psi(F^*)$ is a $K$-minimal exposed face of $\epi f$ and $\Psi^*(\Psi(F^*))=F^*$ whenever $F^*$ is a $K$-minimal exposed face of $\epi f^*$. If $F^*$ is a $K$-minimal exposed face of $\epi f^*$ then, by Proposition \ref{PropExpFaces}, there is some $\bar{x}\in\dom f$ with $\partial f(\bar{x})\neq \emptyset$ such that
\[
F^*=\cb{(u,f^*(u))\in\R^{n}\times \R \st u\in\partial f(\bar{x})}.
\]  
If $u\in \dom f^*$ with $\partial f^*(u)\neq \emptyset$ then $\cb{(x,f(x))\in\R^n\times \R \st u\in\partial f(x)}$ is a $K$-minimal exposed face of $\epi f$. Since the intersection of exposed faces is an exposed face again if it is nonempty (see \cite{Webster} Theorem 2.6.17), $\Psi (F^*)$ is a $K$-minimal exposed face of $\epi f$ whenever $\Psi (F^*)$ is nonempty. But this is true since $(\bar{x},f(\bar{x}))\in\Psi(F^*)$.

Moreover,
\begin{align*}
F^*&=\cb{(u,f^*(u))\in\R^{n}\times \R \st u\in\partial f(\bar x)}\\
&\supseteq\bigcap_{(x,f(x))\in \Psi(F^*)}\cb{(u,f^*(u))\in\R^n\times \R \st u\in\partial f(x)}
=\Psi^*(\Psi(F^*)).
\end{align*}
Next, we show that $F^*\subseteq \Psi^*(\Psi (F^*))$. Assume to the contrary that there is some $(u,f^*(u))\in F^*$ such that $(u,f^*(u))\not\in \Psi^*(\Psi (F^*))$. Hence there is some $(x,f(x))\in\Psi(F^*)$ such that $u\not\in\partial f(x)$. But this contradicts $(u,f^*(u))\in F^*$.

(c) We will show that $\Psi^*(F)$ is a $K$-minimal exposed face of $\epi f^*$ and $\Psi(\Psi^*(F))=F$ whenever $F$ is a $K$-minimal exposed face of $\epi f$. If $F$ is a $K$-minimal exposed face of $\epi f$ then, by Proposition \ref{PropExpFaces}, there is some $\bar u\in\dom f^*$ with $\partial f^*(\bar u)\neq \emptyset$ such that
\[
F=\cb{(x,f(x))\in \R^n\times \R \st \bar u\in\partial f(x)}.
\]
If $x\in \dom f$ with $\partial f({x})\neq \emptyset$ then $\cb{(u,f^*(u))\in \R^n\times \R \st u\in\partial f(x)}$ is a $K$-minimal exposed face of $\epi f^*$. Since the intersection of exposed faces is an exposed face again if it is nonempty, $\Psi^* (F)$ is a $K$-minimal exposed face of $\epi f^*$ whenever $\Psi^*(F)$ is nonempty. But this is true since $(\bar{u},f(\bar{u}))\in\Psi^*(F)$.

Moreover,
\begin{align*}
F&=\cb{(x,f(x))\in\R^{n}\times \R \st \bar u\in\partial f(x)}\\
&\supseteq\bigcap_{(u,f^*(u))\in \Psi^*(F)}\cb{(x,f(x))\in\R^n\times \R \st u\in\partial f(x)}
=\Psi(\Psi^*(F)).
\end{align*}
Now, we show $F\subseteq \Psi(\Psi^* (F))$. Assume to the contrary that there is some $(x,f(x))\in F$ such that $(x,f(x))\not\in \Psi(\Psi^*(F))$. Hence there is some $(u,f^*(u))\in\Psi^*(F)$ such that $u\not\in\partial f(x)$. But this contradicts $(x,f(x))\in F$.

%
%
%
\end{proof}

\section{Second Order Theory}\label{sect:secondorder}

For general proper closed convex functions $f:\R^n\to\OLR$ a property like $\dim F + \dim \Psi^*(F) =n$ as in the piecewise affine case is no longer true as the following example shows.

\begin{example}
Let $f:\R^n\to\R$ be defined by $f(x)=\frac{1}{2}\scp{x,Ax}$ with a symmetric strictly positive definite matrix A. Then $f^*:\R^n\to\R$ is given by $f^*(u)=\frac{1}{2}\scp{u,A^{-1}u}$. Moreover, $\partial f(x)=\cb{Ax}$. Obviously, the faces $F$ of $\epi f$ are exactly the point sets $\cb{\of{x,f(x)}}$ with $x\in \R^n$ and $\Psi^*\cb{\of{x,f(x)}}=\cb{\of{Ax,f^*(Ax)}}$. Hence, $\dim F =\dim \Psi^*(F) =0$ for all faces of $\epi f$. 
\end{example}

In case of smooth functions $f$ and $f^*$ all exposed faces of $\epi f$ and $\epi f^*$ consist of just one point and there exists a duality between the curvatures of $f$ and $f^*$ expressed by the fact that the Hessians of $f$ and $f^*$ are inverse at corresponding points, i.e., if $F=\cb{\of{x,f(x)}}$ then $\Psi^*(F)=\cb{\of{u,f^*(u)}}$ with $u=\nabla f(x)$ and $D^2f^*(u)=\sqb{D^2 f(x)}^{-1}$.

The latter fact was proven by Crouzeix \cite{Crouzeix} and extended by Seeger \cite{Seeger92a} to the case where $f$ and $f^*$ are not necessarily smooth by using a second-order subdifferential.

\subsection{Second-order Subdifferential}

For the definition of the second-order subdifferential we follow mainly \cite[Ch. 13]{RoWe98}. 

For $x,u \in \R^n$ with $f(x)\in\R$ and $t>0$ we define the second-order difference quotient in direction $\xi\in \R^n$ by  
\[
\Delta_t^2f(x|u)(\xi):=\frac{2}{t}\sqb{\frac{f(x+t\xi)-f(x)}{t}-\scp{u,\xi}}
\]
and the corresponding second subderivative by
\[
d^2f(x|u)(\xi):=\liminf_{\substack{t\searrow 0 \\ \xi' \to \xi}}\Delta_t^2f(x|u)(\xi').
\]
Note that $d^2f(x|u)$ is equal to the epigraphical lower limit, i.e., it holds
\[
\epi d^2f(x|u)=\Liminf_{t\searrow 0} \epi\Delta_t^2f(x|u)
\]
where the $\Liminf$ is unerstood in the sense of Painlev\'{e}-Kuratowski.
\begin{definition}
Let $f:\R^n\to\OLR$ and $x,u\in\R^n$ be given with $f(x)\in \R$. $f$ is called {\em twice epi-differentiable} at $x$ relative to $u$ if the functions $\Delta_t^2f(x|u)$ epi-converge to $d^2f(x|u)$ with $t\searrow 0$, i.e., $\epi \Delta_t^2f(x|u)$ converges to $\epi d^2f(x|u)$ in the sense of Painlev\'{e}-Kuratowski.
\end{definition}
The class of twice epi-differentiable functions is rather broad. The following theorem states sufficent conditions for twice epi-differentiability.
\begin{theorem}[\cite{Rockafellar90}, Theorem 3.4.]
Suppose that $f:\R^n\to\OLR$ has the form $f(x)=g(G(x))$ with $G(x)=(G_1(x),...,G_d(x))$, where $g:\R^d\to\OLR$ is convex and piecewise linear-quadratic and the notation is chosen so that the component functions $G_k:\R^n\to \R$ are convex of class $\C^2$ for $k=1,...,p$, but affine for $k=p+1,...,d$. Assume that $g(u)=g(u_1,...,u_d)$ is non-decreasing with respect to the variables $u_1,...,u_p$, and that there exist $\bar x\in\R^n$ and $\bar u\in\dom g$ such that $G_k(\bar x) < \bar u_k$ for $k=1,...,p$ and $g_k(\bar x)=\bar u_k$ for $k=p+1,...,d$. Then $f$ is twice epi-differentiable.
\end{theorem}
If $f$ is proper convex and twice epi-differentiable, $x\in \dom f$ and $u\in\partial f (x)$ then, according to \cite[Prop. 13.20]{RoWe98}, there exists a uniquely defined closed convex set $C\subseteq \R^n$ such that $d^2f(x|u)=\gamma_C^2$ where $\gamma_C$ denotes the gauge function of $C$.
Based on Hiriart-Urruty and Seeger \cite{HUSe89a} we will call this set $C$ the indicatrix of $f$ at $x$ relative to $u$ and denote this set by $\ind f(x|u)$. 
From the theory of gauge functions it follows that 
\[
\ind f(x|u):=\cb{\xi\in\R^n \st d^2f(x|u)(\xi)\le 1}.
\] 
\begin{remark}
In fact Hiriart-Urruty and Seeger \cite{HUSe89a} define upper and lower indicatrices as
\begin{align*}
\overline{\ind}_f(x,u)&=\Limsup_{t\searrow 0}\cb{\xi\in\R^n \st \Delta_t^2f(x|u)(\xi) \le 1}\\
\underline{\ind}_f(x,u)&=\Liminf_{t\searrow 0}\cb{\xi\in\R^n \st \Delta_t^2f(x|u)(\xi) \le 1}.
\end{align*}
They are both subsets of $\ind f(x|u)$ but do not coincide in general. Seeger \cite{Seeger92a} defines $f$ to be {\em second-order regular} at $x$ relative to $u$ if $d^2f(x|u)=\cl \overline{f''}(x,u)$ where $\overline{f''}(x,u)(w)=\limsup_{t\searrow 0}\Delta_t^2f(x|u)(w)$. If $f$ is second-order regular at $x$ relative to $u$ then $d^2f(x|u)(w)=\liminf_{t\searrow 0}\Delta_t^2f(x|u)(w)$ for all $w\in\R^n$ and
\[
\ind f(x|u)=\overline{\ind}_f(x,u)=\underline{\ind}_f(x,u).
\]
In particular, this is the case if $f$ is piecewise linear-quadratic (see \cite[Theorem 3.1]{Rockafellar88}).
\end{remark}

We will now give a geometric interpretation of the indicatrix of a second-order regular function based on the considerations in \cite{Busemann58}, sections 2 and 3.

Given $x_0\in \dom f$ and $\zeta \in\R^n$ with $\norm{\zeta}=1$ we consider the plane $P$ in $\R^n\times \R$ going through the point $(x_0,0)$ spanned by the direction vectors $(\zeta,0)$ and $(0,1)$. The intersection of $P$ with the graph of $f$ is given by the set 
\[
\cb{(x_0+t\zeta,f(x_0+t\zeta)) \st t\in \R,\; x_0+t\zeta \in \dom f}.
\]
Given $m\in\R$, if $x_0+t\zeta\in\dom f$ then let $\rho_t(x_0,\zeta,m)$ be the radius of the circle lying in $P$, going through the points $(x_0,f(x_0))$ and $(x_0+t\zeta,f(x_0+t\zeta))$ and having slope $m$ at $(x_0,f(x_0))$. If $x_0+t\zeta\not\in\dom f$ we define $\rho_t(x_0,\zeta,m)=0$. Then
\[
\rho_t(x_0,\zeta,m)=\sqrt{1+m^2}\of{1+\frac{(f(x_0+t\zeta)-f(x_0))^2}{t^2}}\frac{t}{2}\of{\frac{f(x_0+t\zeta)-f(x_0)}{t}-m}^{-1}.
\]
Let $u\in\partial f(x_0)$ be given then $H_{u}(x_0):=\cb{(x,y)\in \R^n\times \R \st \scp{u,x}-y=\scp{u,x_0}-f(x_0)}$ is a hyperplane supporting $\epi f$ at $(x_0,f(x_0))$.

We define the upper radius of curvature of $f$ at $x_0$ in direction $\zeta$ relative to $u$
as $\bar r(x_0,u,\zeta)=\limsup_{t\searrow 0}\rho_t(x_0,\zeta,\scp{u,\zeta})$ where $\scp{u,w}$ characterizes the slope of the intersection of $H_u(x_0)$ with $P$. We get
\begin{align*}
\liminf_{t\searrow 0}\Delta^2_tf(x_0|u)(\zeta)&=\frac{\sqrt{1+\scp{u,\zeta}^2}\of{1+f'(x_0;\zeta)^2}}{\bar r(x_0,u,\zeta)}.
\end{align*}
If we take into account that $\bar r(x_0,u,\zeta)=0$ if $\scp{u,\zeta}<f'(x_0;\zeta)$,
$\scp{u,\zeta}>f'(x_0;\zeta)$ is impossible due to $u\in\partial f(x_0)$ and that $d^2f(x_0|u)$ 
is positively homogeneous of degree 2 (\cite[Proposition 13.5]{RoWe98}) we can conclude
\[
  \ind f(x_0|u)=\cb{t\zeta\in\R^n \st \norm{\zeta}=1,\; 0\le t \le
  \frac{\sqrt{\bar r(x_0,u,\zeta)}}{\of{1+\scp{u,\zeta}^2}^{\frac{3}{4}}}}.  
\]

Often the polar of the indicatrix is referred to as the second-order subdifferential (see e.g. \cite{HUSe89a,Seeger92a,Seeger94}), i.e., 
\[
\partial^2f(x|u):=(\ind f(x|u))^\circ=\cb{\eta\in\R^n \st \scp{\eta,\xi}\le\sqrt{d^2f(x|u)(\xi)}\text{ for all }\xi\in\R^n}.
\]
Note that the exact definition of the second-order subdifferential varies in the above mentioned papers subject to different convergence concepts that are used in the definition of the second subderivative.

According to \cite[Lemma 4.6]{Seeger92a} (see also \cite[Theorem 13.21]{RoWe98} and the subsequent discussion) the following statement holds.
\begin{theorem}\label{LSO}
Let $f:\R^n \to \OLR$ be a proper closed convex function that is twice epi-differentiable at $x\in\dom f$ relative to $u\in\partial f(x)$. Then $f^*$ is twice epi-differentiable at $u$ relative to $x$ and it holds
\[
\partial^2f^*(u|x)=(\partial^2f(x|u))^\circ=\ind f(x|u)=(\ind f^*(u|x))^\circ.
\]
\end{theorem} 
The next example shows that the preceding lemma is indeed a generalisation of Crouzeix's result.
\begin{example}
If $f$ is twice continuously differentiable at $x$ then $\partial f(x)=\cb{\nabla f(x)}$ and 
\[d^2f(x|\nabla f(x))(\xi)=\scp{D^2f(x)\xi,\xi},\]
where $D^2f(x)$ denotes the Hessian matrix of $f$ at $x$.
Hence 
\[
\ind f(x|\nabla f(x))=\cb{w\in\R^n \st \scp{D^2f(x)\xi,\xi}\le 1}
\]
and
\[
\partial^2 f(x|\nabla f(x))=\cb{\eta\in\R^n \st \scp{\eta,\xi}\le\sqrt{\scp{D^2f(x)\xi,\xi}}\text{ for all }\xi\in\R^n}.
\]
If the Hessian is nonsingular then the subdifferential is a nondegenerate ellipsoid and admits the characterization 
\[
\partial^2 f(x|\nabla f(x))=\cb{\eta\in\R^n \st \scp{(D^2f(x))^{-1}\eta,\eta}\le 1}.
\]
On the other hand, if $f^*$ is twice continuously differentiable as well then
\begin{align*}
\cb{\eta\in\R^n \st \scp{D^2f^*(\nabla x)\eta,\eta}\le 1}&=\ind f^*(\nabla f(x)|x)\\
&=\partial^2 f(x|\nabla f(x))=\cb{\eta\in\R^n \st \scp{(D^2f(x))^{-1}\eta,\eta}\le 1}.
\end{align*}
Hence, it follows from Theorem \ref{LSO} that $D^2f^*(\nabla f(x))=(D^2f(x))^{-1}$.
\end{example}

\subsection{Polyhedral Convex Functions}\label{subsect:PCF}

We consider the case where $f$ is a polyhedral convex function, i.e., $\epi f$ is a polyhedral convex set. $f$ is polyhedral convex if it can be expressed in the form 
\[
f(x)=\max_{i=1,...,m}\sqb{\scp{a_i,x}-b_i}+\delta_D(x)
\]
with $D=\cb{x\in\R^n \st \scp{a_{m+1},x}\le b_{m+1},...,\scp{a_l,x}\le b_l}$.
We assume that none of the affine functions and none of the inequalities can be omitted in the above representation. It is well known that $u\in\partial f(x)$ iff $u\in \conv\cb{a_i \st i\in I(x)}+\cone\cb{a_j \st j\in J(x)}$ where
\[
I(x):=\cb{i\in\cb{1,..,m} \st f(x)=\scp{a_i,x}-b_i},\quad J(x):=\cb{i\in \cb{m+1,...,l} \st \scp{a_i,x}=b_i}.
\]  
According to \cite[Theorem 3.1]{Rockafellar88} we have
$d^2(x|u)=\delta_{K(x,u)}$ with
\begin{align}
K(x,u)&=\cb{w\in\R^n \st \scp{u,w}=f'(x;w)}\notag\\
&=\cb{w\in T_D(x) \st \scp{u,w}=\max_{i\in I(x)}\scp{a_i,w}}\notag\\
&=\cb{w\in T_D(x) \st u\in\conv\cb{a_i\st i\in I'(x,w)}+\cone\cb{a_j:j\in J'(x,w)}}\label{K}
\end{align}
where
\[
T_D(x)=\cb{w\in \R^n \st  \forall j\in J(x): \scp{a_j,w}\le 0}
\]
is the tangent cone to $D$ at $x$ and
\begin{align*}
I'(x,w)&=\cb{i\in I(x) \st \scp{a_i,w}=\max_{j\in I(x)}\scp{a_j,w}}, \quad J'(x,w)&=\cb{j\in J(x) \st \scp{a_j,w}=0}.
\end{align*}
Thus $\ind f(x|u)=\cb{w\in\R^n \st \delta_{K(x,u)}(w)\le 1}=K(x,u)$.

Subsequently we will show that $\ind f(\bar x|\bar u)$ is a linear subspace with $\dim \ind f(\bar x|\bar u) =\dim \Psi(F^*)$ if $\bar x$ and $\bar u$ are chosen such that $(\bar u,f^*(\bar u))\in \ri F^*$ and $(\bar x,f(\bar x))\in \ri \Psi(F^*)$. We start with an auxiliary lemma.
\begin{lemma}\label{LemI'} 
For all $x\in D$, $w \in T_D(x)$ there exists some $\bar t > 0$ such that for all $t\in (0,\bar{t})$ $I'(x,w)=I(x+tw)$ and $J'(x,w)=J(x+tw)$.
\end{lemma}
\begin{proof}
Choose $\bar t >0$ such that 
\[
\bar t \le \frac{\sqb{\scp{a_i,x}-b_i}-\sqb{\scp{a_k,x}-b_k}}{\scp{a_k,w}-\scp{a_i,w}}
\]
if [($i\in I(x)$ and $k\in\cb{1,...,m}\setminus I(x)$) or ($i\in J(x)$ and $k\in\cb{m+1,...,l}\setminus J(x)$)] and $\scp{a_k,w}>\scp{a_i,w}$. Such $\bar t$ exists since $\scp{a_i,x}-b_i>\scp{a_k,x}-b_k$ if ($i\in I(x)$ and $k\in\cb{1,...,m}\setminus I(x)$) or ($i\in J(x)$ and $k\in\cb{m+1,...,l}\setminus J(x)$). 

Let $t\in (0,\bar{t})$ be arbitrarily chosen. 

By the choice of $\bar t$, we have 
\begin{equation}\label{smallerJ}
\forall i\in J(x),\; \forall k\in \cb{m+1,...,l}\setminus J(x) :\quad \scp{a_k,x+tw}-b_k<\scp{a_i,x+tw}-b_i\le 0
\end{equation}
since $w\in T_D(x)$. Hence, $x+tw\in D$ and
\[
i\in J(x+tw)\Leftrightarrow \scp{a_i,x+tw}=b_i \Leftrightarrow \of{i\in J(x) \text{ and }\scp{a_i,w}=0}
\Leftrightarrow i\in J'(x,w).
\]
Analogously, $\scp{a_i,x+tw}-b_i >\scp{a_k,x+tw}-b_k$ for all $i\in I(x)$ and all $k\in \cb{1,...,m}\setminus I(x)$.
Consequently, we have 
\begin{align*}
i\in I(x+tw)&\Leftrightarrow \scp{a_i,x+tw}-b_i=f(x+tw)=\max_{j=1,...,m}\sqb{\scp{a_j,x+tw}-b_j}\\
&\Leftrightarrow \of{i\in I(x) \text{ and } \scp{a_i,w}=\max_{j\in I(x)}\scp{a_j,w}} \Leftrightarrow i\in I'(x,w) .
\end{align*}
\end{proof}
%
\begin{corollary}\label{CorK}
\[
\ind f(x|u)=\cb{w\in\R^n \st \exists\bar{t}>0\;\forall  t\in (0,\bar{t}) : u\in\partial f(x+tw)}.
\]
\end{corollary}
\begin{proof}
''$\subseteq$:'' Let $w\in \ind f(x|u)$, i.e., $w\in T_D(x)$ and 
\[
u\in\conv\cb{a_i \st i\in I'(x,w)}+\cone\cb{a_j \st j\in J'(x,w)}.
\]
By Lemma \ref{LemI'} there is some $\bar t>0$ such that for all $t\in (0,\bar t)$,
\[
u\in\conv\cb{a_i \st i\in I(x+tw)}+\cone\cb{a_j \st j\in J(x+tw)}
\]
hence $u\in \partial f(x+tw)$.

''$\supseteq$:'' Assume that there is some $\bar t >0$ with $u\in \partial f(x+tw)$ for all $t\in (0,\bar{t})$. Then 
\[
u\in\conv\cb{a_i \st i\in I(x+tw)}+\cone\cb{a_j \st j\in J(x+tw)}
\] 
for all $t\in (0,\bar{t})$,  i.e., 
\[
u\in\conv\cb{a_i \st i\in I'(x,w)}+\cone\cb{a_j \st j\in J'(x,w)}
\]
by Lemma \ref{LemI'}.
Moreover, $w\in T_D(x)$ since otherwise $x+tw\not \in D$ for all $t>0$ contradicting $\partial f(x+tw)\neq \emptyset$.
\end{proof}

If $f$ is polyhedral convex, $f^*$ is polyhedral convex, too, i.e., it can be expressed as
\[
f^*(u)=\max_{i=1,...,p}\sqb{\scp{a_i^*,u}-b_i^*}+\delta_{D^*}(u)
\]
with $D^*=\cb{u\in\R^n \st \scp{a_{p+1}^*,u}\le b_{p+1}^*,...,\scp{a_q^*,x}\le b_q^*}$.
Each $K$-minimal proper (exposed) face $F^*$ of $\epi f^*$ is uniquely characterized by a pair of index sets $I^*\subseteq \cb{1,...,p}$ and $J^*\subseteq \cb{p+1,...,q}$ (where $I^*$ must be nonempty and $J^*$ may be empty) in the following way
\[
F^*=\cb{(u,f^*(u)) \st u\in D^*,\; \forall j\in J^* : \scp{a_j^*,u}=b_j^*,\; \forall i\in I^* : f^*(u)=\scp{a_i^*,u}-b_i^*}.
\]
Moreover, $(\bar u,f^*(\bar u))\in \ri F^*$ iff $I^*(\bar u)=I^*$ and $J^*(\bar u)=J^*$ where 
\[
I^*(u):=\cb{i\in\cb{1,..,p} \st f^*(u)=\scp{a_i^*,u}-b_i^*},\quad J^*(u):=\cb{i\in \cb{p+1,...,q} \st \scp{a_i^*,u}=b_i^*}.
\]
Let $(\bar u,f^*(\bar u))\in \ri F^*$ then
\begin{align*}
\Psi(F^*)&=\bigcap_{(u,f^*(u))\in F^*} \cb{(x,f(x)\in\R^n\times \R \st x\in \partial f^*(u)}\\
&=\bigcap_{(u,f^*(u))\in F^*} \cb{(x,f(x))\in\R^n\times \R \st x\in \conv\cb{a_i^*\st i\in I^*(u)}+\cone\cb{a_i^*\st i\in J^*(u)}}\\
&=\cb{(x,f(x))\in\R^n\times \R \st x\in \conv\cb{a_i^*\st i\in I^*(\bar u)}+\cone\cb{a_i^*\st i\in J^*(\bar u)}}\\
&=\cb{(x,f(x))\in\R^n\times \R \st x\in \partial f^*(\bar u)}\\
&=\cb{(x,f(x))\in\R^n\times \R \st \bar u \in \partial f(x)}
\end{align*}
From Corollary \ref{CorK} we conclude
\[\begin{split}
w\in \ind f(x|\bar{u}) &\Leftrightarrow \exists\bar{t}>0\; \forall  t\in (0,\bar{t}) : \bar u\in\partial f(x+tw)\\
&\Leftrightarrow \exists\bar{t}>0\; \forall  t\in (0,\bar{t}) : (x+tw,f(x+tw))\in \Psi(F^*).
\end{split}\]
From this representation it is easy to see that $\ind f(\bar x|\bar u)$ is a linear subspace of $\R^n$ with $\dim\ind f(\bar x|\bar u)=\dim\Psi(F^*)$ if $(\bar x,f(\bar x))\in \ri \Psi(F^*)$.


Analogously, we can show  that $\ind f^*(\bar{u}|\bar{x})$ is a linear subspace of $\R^n$, too, with $\dim \ind f^*(\bar{u}|\bar{x})=\dim F^*$.
Since $\ind f(\bar{x}|\bar{u})$ and $\ind f^*(\bar{u}|\bar{x})$ are linear subspaces of $\R^n$ that are polar to each other by Theorem \ref{LSO}, their dimensions add up to $n$. Hence $\dim F^* + \dim \Psi(F^*)=n$ follows from Theorem \ref{LSO}.

\section{Application to Vector Optimization}\label{sect:vectopt}

In the previous sections we have shown geometric duality relations between the epigraph of a proper closed convex function and the epigraph of its conjugate. In this section we will describe a transformation of the extended image of a vector optimization problem into an epigraph of a proper closed convex function $f$ and we will determine the dual problem.

Let $\Gamma: \R^m \to R^q$ be a vector-valued objective function that has to be minimized over a nonempty convex feasible set $\X\subseteq\R^m$ with respect to the ordering generated by a nonempty closed convex cone $C\subseteq\R^q$ that is not a linear space.
We assume that $\Gamma$ is $C$-convex, i.e., for all $x_1,x_2\in \R^m$, $t\in[0,1]$ we have
\[
(1-t)\Gamma(x_1)+t\Gamma(x_2)-\Gamma\of{(1-t)x_1+tx_2} \in C.
\]
We want to derive geometric duality relations for the upper closed extended image $\P:=\cl\of{\Gamma[\X]+C}$ of this vector optimization problem. Obviously, $\P$ is closed and $C$-convexity of $\Gamma$ implies convexity of $\P$. 

We are going to construct a linear transformation $T$ and a proper convex function $f:\R^{q-1} \to \OLR$ such that $\P = T[\epi f]$. To this end,  
let $k\in\ri C$ and $e^1,...,e^{q-1}$ be vectors in $\R^q$ such that $e^1,...,e^{q-1},k$ are linearly independent. Let $T:=(e^1,...,e^{q-1}, k)$ be the nonsingular matrix formed by these vectors and $E:=(e^1,...,e^{q-1})$. 
Let 
\[
\varphi(y):=\inf\cb{r\in\R \st rk-y\in C}
\]
and
\[
f(z):=\inf_{x\in\X}\varphi(\Gamma(x)-Ez). 
\]
Note that the function $\varphi$ is a well known scalarization functional in vector optimization that has a wide range of applications. Hamel \cite{HamelTSF} has written a nice survey about history and properties of this kind of functional.
It is not hard to show that $\varphi$ is a lower semicontinuous sublinear function (see e.g. \cite{HamelTSF}). Moreover, $\varphi$ is proper by the following lemma and $\dom \varphi = \R\cb{k}-C\neq \emptyset$.
\begin{lemma}
$\varphi(y)\neq -\infty$ for all $y\in \R^q$.
\end{lemma}
\begin{proof}
We have $-k\not \in C$ since otherwise $0=k-k\in \ri C + C =\ri C$ implying that $C$ is a linear space.
Consequently, we can strongly separate $C$ and $-k$, i.e., there are $v\in\R^q\setminus\cb{0}$, $\alpha < 0$ such that $\scp{v,c} > \alpha > \scp{v,-k}$ for all $c\in C$. Let $y\in\R^q$ and $r:=\of{\scp{v,y}+\alpha}/\scp{v,k}$ then $\scp{v,rk-y}=\alpha$, i.e., $rk-y\not\in C$ implying $\varphi(y)\neq -\infty$.
\end{proof}

The following equivalent descriptions of $f$ will turn out to be useful in the sequel.
\begin{lemma}
\[
f(z)=\inf\cb{r\in\R \st T\begin{pmatrix}z \\ r\end{pmatrix}\in \P}=\inf_{y\in\P}\varphi\of{y-Ez}.
\]
\end{lemma}
\begin{proof}
\begin{align*}
f(z)&=\inf_{x\in\X}\varphi(\Gamma(x)-Ez)=\inf\cb{r\in \R \st rk-\Gamma(x)+Ez \in C,\; x\in \X}\\
&=\inf\cb{r\in \R \st T\begin{pmatrix}z \\ r\end{pmatrix}\in \Gamma[\X]+C}=\inf\cb{(T^{-1})_q(y) \st y\in\Gamma[\X]+C} \\
&= \inf\cb{(T^{-1})_q(y) \st y\in\P}=\inf\cb{r\in \R \st T\begin{pmatrix}z \\ r\end{pmatrix}\in \P} \\
&=\inf\cb{r\in \R \st Ez + rk \in \P} = \inf\cb{r\in \R \st Ez + rk \in \P + C} \\
&= \inf\cb{r\in \R \st Ez + rk \in \cb{y} + C,\; y\in \P}=\inf_{y\in \P}\varphi\of{y-Ez}.
\end{align*}
\end{proof}

Let $R(z):=\cb{r\in\R \st T\begin{pmatrix}z \\ r\end{pmatrix}\in \P}$. For every $z\in \R^{q-1}$, $R(z)$ is closed since $\P$ is closed and $R(z)$ is an upper set, i.e., $r\in R(z)$ and $r'\ge r$ imply $r'\in R(z)$ due to the definition of $\P$ and $k\in C$. Hence
\[
r\ge f(z) \Leftrightarrow r\in R(z) \Leftrightarrow T\begin{pmatrix}z \\ r\end{pmatrix}\in \P,
\]
i.e., $\epi f =T^{-1}(\P)$ or $\P = T[\epi f]$. Consequently, $f$ is a closed convex function since $\P$ is closed convex and $T$ is linear and continuous. Moreover, $f$ is proper iff $\P$ is nontrivial, i.e., $\emptyset \neq \P\neq \R^q$.

A point $(z,r)\in \epi f$ is minimal with respect to $K=\cb{(z,r)\in \R^{q-1}\times\R \st z=0,\;r\ge 0}$ if and only if $r=f(z)$. Next, we will analyze the minimality properties of the transformed points $T\begin{pmatrix}z \\ f(z)\end{pmatrix}$ in $\P$ with respect to $C$. The following notion turns out to be useful.
\begin{definition}
A point $y\in\P$ is said to be {\em relatively minimal} in $\P$ with respect to $C$ iff $\P\cap\of{\cb{y}- \ri C}=\emptyset$. 

The set of all relatively minimal points in $\P$ with respect to $C$ is denoted by $\rMin_C \P$.
\end{definition}
Note that the set of relative minimal points coincides with the set of weakly minimal points if $C$ has nonempty interior and it coincides with the set of minimal points if $C$ is just a ray.

\begin{proposition}
$y\in\rMin_C \P$ if and only if there is some $z\in\R^{q-1}$ such that 
\[y=T\begin{pmatrix}z \\ f(z)\end{pmatrix}.
\]
\end{proposition}
\begin{proof}
First, assume that $y=T\begin{pmatrix}z \\ f(z)\end{pmatrix}=Ez+f(z)k$ and $y\not\in \rMin_C \P$. Then there is some $y'\in \P$ such that $y-y'\in \ri C$. $y-y' \in \ri C$ and $k\in C$ imply the existence of some $\mu > 1$ such that $(1-\mu)k +\mu(y-y')\in C$, i.e., $y-y'-tk \in C$ with $t=(\mu-1)/\mu > 0$. 
Hence, we get $C \ni y-y'-tk=Ez-y'+(f(z)-t)k$, i.e., $\varphi (y'-Ez)\le f(z)-t < f(z)$ contradicting $f(z)=\inf_{y\in\P}\varphi\of{y-Ez}$.

On the other hand, let $\bar{y}\in \rMin_C \P$ and $\begin{pmatrix}\bar z \\ \bar r\end{pmatrix}:=T^{-1}\bar y$, i.e., $\bar y =E\bar z + \bar r k$. We will show that $\bar r = f(\bar z)$. Assume to the contrary that 
\[
\inf_{y\in\P}\varphi\of{y-E\bar z} = f(\bar z) <\bar r.
\]
Then there is some $y \in \P$ with $\varphi\of{y-E\bar z} < \bar r$ and, by definition of $\varphi$, some $r\in \R$ with $rk - y  + E\bar z \in C$ and $r < \bar r$. Hence
\begin{align*}
\bar y - y = E\bar z +\bar r k - y = (rk-y+E\bar z)+(\bar r -r)k\in\ri C
\end{align*}
since $C$ is a convex cone and $k\in \ri C$. But $\bar y - y\in \ri C$ contradicts $\bar{y}\in \rMin_C \P$.
\end{proof}

Next, we give a characterization of the conjugate of $f$. 
\begin{proposition}
\[
f^*(w)=\begin{cases}\sqb{\scp{c^*(-w),\Gamma}+\delta_\X}^*(0) & \text{if }c^*(-w)\in C^+ \\ \infty & \text{otherwise}\end{cases}
\]
where $c^*(w):=T^{-T}\begin{pmatrix} w \\ 1 \end{pmatrix}$ and $C^+:=\cb{c^*\in\R^q \st \forall c\in C : \scp{c^*,c} \ge 0}$ is the positive dual cone of $C$.
\end{proposition} 
\begin{proof}
We have 
\[
f(z)=\inf_{x\in\R^q}\Phi(x,z)\quad \text{ with }\quad \Phi(x,z):=\varphi(M(x)-Ez)+\delta_\X(x)
\]
hence $f^*(w)=\Phi^*(0,w)$ (\cite[Theorem 2.6.1]{Zalinescu}). We apply \cite[Theorem 2.8.10]{Zalinescu} in order to calculate $\Phi^*$. It is not hard to show that all assumptions of the theorem are satisfied, in particular, we have $D=\R^q$ hence condition (vi) is satisfied. Thus, we obtain.
\[
\Phi^*(0,w)=\min\cb{\sqb{\scp{c^*,\Gamma(x)-Ez}+\delta_{\X}(x)}^*(0,w)+\varphi^*(c^*)\st c^*\in C^+}.
\]
Moreover, we have 
\[
\varphi^*(c^*)=\begin{cases}0 &\text{ if } c^*\in C^+,\;\scp{c^*,k}=1 \\ \infty & \text{ otherwise.} \end{cases}
\]
(see e.g. \cite[Corollary 9]{HamelTSF}). Hence we get
\[\begin{split}
f^*(w)&=\min\cb{\sqb{\scp{c^*,\Gamma(x)-Ez}+\delta_{\X}(x)}^*(0,w)\st c^*\in C^+,\;\scp{c^*,k}=1}\\
&=\min\cb{\sqb{\scp{c^*,\Gamma}+\delta_{\X}}^*(0)+\scp{-c^*,E}^*(w)\st c^*\in C^+,\;\scp{c^*,k}=1}\\
&=\min\cb{\sqb{\scp{c^*,\Gamma}+\delta_{\X}}^*(0)\st c^*\in C^+,\;k^Tc^*=1,\;E^Tc^*=-w}\\
&=\begin{cases}\sqb{\scp{c^*(-w),\Gamma}+\delta_{\X}}^*(0) & \text{if }c^*(-w)=T^{-T}\begin{pmatrix} -w \\ 1 \end{pmatrix}\in C^+ \\ \infty & \text{otherwise.}\end{cases}
\end{split}\]
\end{proof}

If we define the dual image $\D$ by
\[
\D:=\cb{\begin{pmatrix}w\\s\end{pmatrix}\st c^*(w)\in C^+,s\le -\sqb{\scp{c^*(w),\Gamma}+\delta_{\X}}^*(0)}
\]
then $\D=-\epi f^*$.
\begin{remark}
Note that 
\[
-\sqb{\scp{c^*(w),\Gamma}+\delta_\X}^*(0)=\inf_{x\in\X}\scp{c^*(w),\Gamma(x)},
\]
is nothing else than the optimal value of the primal problem scalarized by the linear funtional $c^*(w)$.
\end{remark}

The linear transformation $T$ provides a one-to-one correspondence between $K$-minimal exposed faces of $\epi f$ and relatively $C$-minimal exposed faces of $\P$ and the mapping $-\id$ sets a one-to-one correspondence between $K$-maximal exposed faces of $\D$ and $K$-minimal exposed faces of $\epi f^*$.

Hence, by Theorem \ref{TGD}, $\widehat{\Psi}$ defined by $\widehat{\Psi}(F^*)= T[\Psi(-F^*)]$ defines an inclusion reversing one-to-one mapping between $K$-maximal exposed faces of $\D$ and relatively $C$-minimal exposed faces of $\P$
with inverse mapping $\widehat{\Psi}^{-1}=-\Psi^*\circ T^{-1}$.

For $y,v\in\R^q$ we define 
\[\psi(y,v):=(v_1,...,v_{q-1},1)T^{-1}y-v_q,
\]
\[
H(v):=\cb{y\in\R^q \st \psi(y,v)=0}\quad \text{and}\quad H^*(y):=\cb{v\in\R^q \st \psi(y,v)=0}.
\]

Then we obtain 
\begin{align}
\widehat{\Psi}(F^*)&=T[\Psi (-F^*)]=\bigcap_{(w,f^*(w))\in -F^*}T\sqb{\cb{(z,f(z))\in\R^q \st w\in\partial f(z)}} \label{Psihat}\\
&=\bigcap_{(w,f^*(w))\in -F^*}T\sqb{\cb{(z,r)\in\R^q \st r=f(z),\; \scp{w,z} = f(z) +f^*(w)}}\nonumber\\
&=\bigcap_{(w,f^*(w))\in -F^*}T\sqb{\cb{(z,r)\in\epi f \st \scp{w,z} = r +f^*(w)}}\nonumber\\
&=\bigcap_{v\in F^*}T\sqb{\cb{(z,r)\in\epi f  \st r - v_q + (v_1,...,v_{q-1})z = 0}}\nonumber\\
&=\bigcap_{v\in F^*}T\sqb{\cb{(z,r)\in\epi f  \st (v_1,...,v_{q-1},1)\begin{pmatrix}z\\r\end{pmatrix}-v_q =0}}\nonumber\\
&=\bigcap_{v\in F^*}\cb{y \in \P \st \psi(y,v)=0}\nonumber\\
&=\bigcap_{v\in F^*}H(v)\cap\P\nonumber
\end{align}
and
\[\begin{split}
\widehat{\Psi}^{-1}(F)&=-\Psi^*(T^{-1}[F])=\bigcap_{(z,f(z))\in T^{-1}[F]}-\cb{(w,f^*(w))\in\R^q \st w\in\partial f(z)}\\
&=\bigcap_{(z,f(z))\in T^{-1}[F]}-\cb{(w,s)\in\R^q \st s=f^*(w),\; \scp{w,z}=f(z)+f^*(w)}\\
&=\bigcap_{(z,f(z))\in T^{-1}[F]}-\cb{(w,s)\in\epi f^* \st -\scp{w,z}+f(z)+s=0}\\
&=\bigcap_{y\in F}-\cb{(w,s)\in\epi f^* \st (-w^T,1)T^{-1}y+s =0}\\
&=\bigcap_{y\in F}\cb{v\in\D \st \psi(y,v)=0}\\
&=\bigcap_{y\in F}H^*(y)\cap \D.
\end{split}\]

In order to obtain second order relations between the sets $\P$ and $D$ from Theorem \ref{LSO} we define indicatrices for these sets. The indicatrix for the set $\P$ at a point $\bar y\in\rMin_C \P$ relative to a normal vector $\eta \in\mathcal{N}_\P(\bar y)$ with $k^T\eta=-1$ depending on the transformation $T$ will be defined by
\[
\ind\nolimits_{\P,T} (\bar y|\eta):=\ind f\of{\widetilde{T^{-1}}\bar y|E^T\eta}
\]
where $\widetilde{T^{-1}}$ collects the first $q-1$ rows of the matrix $T^{-1}$. Note further that $E^T\eta\in \partial f(\widetilde{T^{-1}}\bar y)$ if $\eta \in\mathcal{N}_\P(\bar y)$ with $\scp{\eta,k}=-1$. This can be derived from the fact that 
\begin{equation}\label{n}
w\in \partial f(\widetilde{T^{-1}}\bar y) \quad\Leftrightarrow\quad \begin{pmatrix} w \\ -1 \end{pmatrix} \in \mathcal{N}_{\epi f}(T^{-1}\bar y) \quad\Leftrightarrow\quad T^{-T} \begin{pmatrix} w \\ -1 \end{pmatrix} \in \mathcal{N}_{\P}(\bar y).
\end{equation}
If $f$ is second-order regular and $T$ is an orthogonal matrix, i.e., the transformation $T$ is angle- and length-preserving, a geometric interpretation of $\ind\nolimits_{\P,T} (\bar y|\eta)$ can be given similarly to that of section 3. Let $\zeta$ be a unit vector in $\R^{q-1}$. Then we consider the plane $P$ going through the point $\bar y$ spanned by the direction vectors $k$ and $E\zeta$. Let $\rho_t(\zeta)$ be the circle in $P$ through the points $\bar y$ and $y_t:=T\begin{pmatrix} \widetilde{T^{-1}}\bar y+t\zeta \\ f\of{\widetilde{T^{-1}}\bar y+t\zeta} \end{pmatrix}$ having as a tangent at $\bar y$ the intersection of $P$ with the hyperplane $\cb{y\in\R^q \st \scp{\eta,y-\bar y}=0}$ and let $\bar r (\zeta):=\limsup_{t\searrow 0}\rho_t(\zeta)$. Then
\[
\ind\nolimits_{\P,T} (\bar y|\eta)=\cb{t\zeta\in\R^{q-1} \st \norm{\zeta}=1,\; 0\le t \le
  \frac{\sqrt{\bar r(\zeta)}}{\of{1+\scp{\eta,E\zeta}^2}^{\frac{3}{4}}}}. 
\]

Analogously, for $\bar v\in\Min_K \D$ and $\eta^*\in\mathcal{N}_{\D}(\bar v)$ with $\eta^*_q=1$ we can define
\[
\ind\nolimits_{\D,-\id} (\bar v|\eta^*) = \ind f^*\of{-(\bar v_1,...,\bar v_{q-1})^T |-(\eta^*_1,...,\eta^*_{q-1})^T}
\]
since 
\begin{equation}\label{n*}\begin{split}
-(\eta^*_1,...,\eta^*_{q-1})^T\in \partial f^* \of{-(\bar v_1,...,\bar v_{q-1})^T} 
&\quad\Leftrightarrow\quad -(\eta^*_1,...,\eta^*_{q-1},1)^T\in \mathcal{N}_{\epi f^*}(-\bar v) \\
&\quad\Leftrightarrow\quad (\eta^*_1,...,\eta^*_{q-1},1)^T\in \mathcal{N}_{\D}(\bar v).
\end{split}\end{equation}

Given $v\in F^*$ and $y\in\widehat{\Psi}(F^*)$ for some $K$-maximal proper exposed face $F^*$ of $\D$ one can easily derive from equations \eqref{Psihat}, \eqref{n} and \eqref{n*} that
\[
\eta(v):=-c^*\of{\begin{pmatrix}v_1 \\ \vdots \\ v_{q-1}\end{pmatrix}}=-T^{-T}\begin{pmatrix} v_1 \\ \vdots \\v_{q-1} \\ 1 \end{pmatrix} \in \mathcal{N}_\P(y)
\]
and
\[
\eta^*(y):=\begin{pmatrix}-\widetilde{T^{-1}}y \\ 1 \end{pmatrix}\in \mathcal{N}_\D(v).
\]
Moreover, $\scp{\eta(v),k} = -1$ and $\eta^*_q(y)=1$.  

The above considerations and the results from sections \ref{sect:epi} and \ref{sect:secondorder} lead to the following theorem.
\begin{theorem}\label{thm:MTVO}
The mapping $\widehat{\Psi} : 2^{\R^q} \to 2^{\R^q}$ defined by 
\[
\widehat{\Psi}(F^*)=\bigcap_{v\in F^*}H(v)\cap\P
\]
is an inclusion-reversing one-to-one map between the set of all $K$-maximal exposed faces of $\D$  and the set of all relatively $C$-minimal exposed faces of $\P$ with inverse
\[
\widehat{\Psi}^{-1}(F)=\bigcap_{y\in F}H^*(y)\cap \D.
\]
Moreover, if $f$ is twice epi-differentiable then for every $K$-maximal exposed face $F^*$ of $\D$,
\begin{equation}\label{eq:ind}
\forall v\in F^*, \forall y\in \widehat{\Psi}(F^*):\quad \ind\nolimits_{\D,-\id} (v|\eta^*(y))=\sqb{\ind\nolimits_{\P,T} (y|\eta(v))}^{\circ}
\end{equation}
holds true.
If $\P$ is polyhedral then $\D$ is polyhedral as well and \eqref{eq:ind} implies
\[
\dim F^* + \dim \hat{\Psi}(F^*)=q-1.
\]
\end{theorem}

\begin{example}
We consider the special case of a linear vector optimization problem, i.e., $\Gamma$ is a linear operator and $\X=\cb{x\in\R^m \st Ax \ge b}$ with $A\in \R^{p\times m}$, $b\in \R^p$ for some $p\in \N$. Moreover, we assume that $k\in \ri C$ can be chosen such that $k_q=1$.

Then we can chose 
\[
E=\begin{pmatrix} 1 & \cdots & 0 \\ \vdots & \ddots & \vdots \\ 0 & \cdots & 1 \\ 0 & \cdots & 0 \end{pmatrix} \text{ i.e. }\quad T=\begin{pmatrix} 1 & \cdots & 0 & k_1 \\ \vdots & \ddots & \vdots & \vdots\\ 0 & \cdots & 1 & k_{q-1} \\ 0 & \cdots & 0 & 1 \end{pmatrix} \text{ and }\quad
T^{-1}=\begin{pmatrix} 1 & \cdots & 0 & -k_1 \\ \vdots & \ddots & \vdots & \vdots\\ 0 & \cdots & 1 & -k_{q-1} \\ 0 & \cdots & 0 & 1 \end{pmatrix}.
\]
Hence
\begin{equation}\label{eq:cw}
c^*(w)=T^{-T}\begin{pmatrix} w \\ 1 \end{pmatrix}=\begin{pmatrix} w \\ 1-\sum_{i=1}^{q-1} k_i w_i \end{pmatrix}.
\end{equation}
From duality for scalar linear optimization problems we obtain
\[
-\sqb{\scp{c^*(w),\Gamma}+\delta_\X}^*(0)=\inf\cb{c^*(w)^T\Gamma(x) \st Ax \ge b} =\sup\cb{b^T u \st u\ge 0, A^Tu=\Gamma^Tc^*(w)}
\]
where the supremum is either a maximum or $-\infty$. $+\infty$ is impossible since we assume that $\X$ is nonempty. Thus we obtain
\begin{align*}
\D&=\cb{\begin{pmatrix}w\\s\end{pmatrix}\st c^*(w)\in C^+,\;s\le -\sqb{c^*(w)^T\Gamma+\delta_{\X}}^*(0)} \\
&=\cb{\begin{pmatrix}w\\s\end{pmatrix}\st c^*(w)\in C^+,\;s\le \sup\cb{b^T u \st  u\ge 0, A^Tu=\Gamma^Tc^*(w)}}\\
&=\cb{\begin{pmatrix}c^*_1 \\ \vdots \\ c^*_{q-1}\\b^T u -r\end{pmatrix}\st c^*\in C^+,\; u\ge 0,\; A^Tu=\Gamma^Tc^*,\; k^Tc^* =1,\; r\ge 0}
\end{align*}
since we can solve \eqref{eq:cw} for $w$ by $w=(c^*_1,...,c^*_{q-1})^T$ iff $c^*_q = 1-\sum_{i=1}^{q-1} k_i c^*_i$, i.e., $k^Tc^*=1$.

Moreover, we obtain
\begin{align*}
\psi(y,v)&=(v_1,...,v_{q-1},1)T^{-1}y-v_q\\
&=\sum_{i=1}^{q-1}v_i(y_i-k_iy_q) + y_q - v_q\\
&=\sum_{i=1}^{q-1}y_iv_i + y_q\of{1-\sum_{i=1}^{q-1}k_iv_i} - v_q.
\end{align*}
\end{example}

\bibliographystyle{abbrv}
\bibliography{GD_general}

%
%
%
%
%
%
%
%
%
%

\end{document}